\documentclass{article}

\usepackage{amsmath,amssymb,amsthm}

\usepackage{a4wide}

\newtheorem{theorem}{Theorem}

\newtheorem{lemma}{Lemma}

\newtheorem{ex}{Example}
\theoremstyle{remark}
\newtheorem{remark}{Remark}

 \theoremstyle{definition}
\newtheorem{definition}{Definition}

\title{Transversality Conditions for Infinite Horizon\\
Variational Problems on Time Scales\thanks{Submitted 6-October-2009;
Accepted 19-March-2010 in revised form;
for publication in \emph{Optimization Letters}.}}

\author{Agnieszka B. Malinowska$^{1, 2}$\\
\texttt{abmalinowska@ua.pt}
\and Nat\'{a}lia Martins$^{1}$\\
\texttt{natalia@ua.pt}
\and Delfim F. M. Torres$^{1}$\\
\texttt{delfim@ua.pt}}

\date{\begin{minipage}[t]{0.45\linewidth}
\centering
$^1$Department of Mathematics\\
University of Aveiro\\
3810-193 Aveiro, Portugal
\end{minipage}
\begin{minipage}[t]{0.45\linewidth}
\centering
$^2$Faculty of Computer Science\\
Bia{\l}ystok University of Technology\\
15-351 Bia\l ystok, Poland
\end{minipage}
}

\begin{document}

\maketitle

\vspace*{-0.4cm}

\begin{abstract}
We consider problems of the calculus of variations
on unbounded time scales. We prove the validity of the
Euler-Lagrange equation on time scales for infinite horizon
problems, and a new transversality condition.

\smallskip


\smallskip

\end{abstract}


\section{Introduction}

Starting with Ramsey's pioneering work \cite{Ramsey}, infinite
horizon variational and optimal control problems have been widely
used in economics (see, \textrm{e.g.},
\cite{Beavis,Cai1,Cai2,Cai3,Chiang,Dasgupta,G:M:P:01,P:R:K:10}).
On the other hand, economics is a discipline
in which there appears to be many opportunities for applications of time scales
\cite{Ric:Del:JVC,A:B:L:06,Atici:comparison,A:U:08,Basia:post_doc_Aveiro:2}.
Therefore, it is natural to try to relate the subject of time scales
with the study of infinite horizon variational problems.
This is the main motivation and goal of the present work.

The study of improper integrals on time scales
was introduced by Bohner and Guseinov in
\cite{Bohner-Guseinov 2003}. For a more general
treatment see \cite{diamond}. However, the
use of time scale integrals in the calculus of variations
has been, so far, restricted to bounded intervals --
see \cite{B:T:08,BohnerCV04,B:F:T:10,comRui:TS:Lisboa07,M:T:2010,NataliaHigherOrderNabla}
and references therein. In this paper we consider
the infinite horizon problem of maximizing the expression
\begin{equation}\label{int}
\int_{a}^{T} L(t,x^{\sigma}(t), x^{\Delta}(t)) \Delta t
\end{equation}
as $T$ grows to infinity. If $T=+\infty$, then the integral
\eqref{int} does not necessarily converge. It may diverge to plus or
minus infinity or it may oscillate.
In such situations the extension of the definition of
optimality used in the time scale setting (see \cite{BohnerCV04})
to the unbounded time domain is not very useful.
For example, if every admissible function $x$ yields an infinite value for
functional \eqref{int}, then each admissible path could be called an
optimal path. To handle this and similar situations in a rigorous
way, several alternative definitions of optimality for problems with
unbounded time domain have been proposed in the literature (see,
\textrm{e.g.}, \cite{Brock,Gale,SS,Weiz}). In this paper we follow
Brock's notion of optimality. Therefore, our optimality criterion
for the special case $\mathbb{T}=\mathbb{Z}$ coincides with Brock's
notion of weak maximality \cite{Brock}. If $\mathbb{T}=\mathbb{R}$,
our definition of weak maximality coincides with the extension of
Brock's notion of weak maximality to the continuous time situation
\cite{Papageorgiou}. It is worth to point out that in the case where
the functional \eqref{int} converges for all admissible paths, the
weak maximal path is optimal in the sense of the standard
definition of optimality.

Main result of the paper gives necessary conditions of
weak maximality for infinite horizon variational
problems on a generic (unbounded) time scale
(\textrm{cf.} Theorem~\ref{main result}).


\section{Preliminaries}

In this section we introduce basic definitions and results that will
be needed in the sequel. For a more general presentation of the
theory of time scales we refer the reader to the books
\cite{Bohner-Peterson 2001,Bohner-Peterson 2003}.
As usual,  $\mathbb{R}$, $\mathbb{Z}$, and $\mathbb{N}$ denote,
respectively, the set of real, integer, and natural numbers.

A {\it time scale} $\mathbb{T}$ is an arbitrary nonempty closed
subset of  $\mathbb{R}$.  Besides standard cases of $\mathbb{R}$
(continuous time) and $\mathbb{Z}$ (discrete time), many different
models of time are used. For each time scale $\mathbb{T}$ the
following operators are used:

\begin{itemize}
\item the {\it forward jump operator} $\sigma:\mathbb{T} \rightarrow \mathbb{T}$,
$\sigma(t):=\inf\{s \in \mathbb{T}:s>t\}$ for $t\neq\sup \mathbb{T}$
and $\sigma(\sup\mathbb{T})=\sup\mathbb{T}$ if
$\sup\mathbb{T}<+\infty$;

\item the {\it backward jump operator} $\rho:\mathbb{T} \rightarrow \mathbb{T}$,
$\rho(t):=\sup\{s \in \mathbb{T}:s<t\}$ for $t\neq\inf \mathbb{T}$
and $\rho(\inf\mathbb{T})=\inf\mathbb{T}$ if
$\inf\mathbb{T}>-\infty$;

\item the {\it forward graininess function} $\mu:\mathbb{T} \rightarrow [0,\infty[$,
$\mu(t):=\sigma(t)-t$.
\end{itemize}

\begin{ex}
If $\mathbb{T}=\mathbb{R}$, then for any $t \in \mathbb{R}$,
$\sigma(t)=t=\rho(t)$ and $\mu(t) \equiv 0$. If
$\mathbb{T}=\mathbb{Z}$, then for every $t \in \mathbb{Z}$,
$\sigma(t)=t+1$, $\rho(t)=t-1$ and $\mu(t) \equiv 1$.
\end{ex}

A point $t\in\mathbb{T}$ is called \emph{right-dense},
\emph{right-scattered}, \emph{left-dense} and \emph{left-scattered}
if $\sigma(t)=t$, $\sigma(t)>t$, $\rho(t)=t$, and $\rho(t)<t$,
respectively. We say that $t$ is \emph{isolated} if
$\rho(t)<t<\sigma(t)$, that $t$ is \emph{dense} if
$\rho(t)=t=\sigma(t)$.
If $\sup \mathbb{T}$ is finite and left-scattered, we define
$\mathbb{T}^\kappa :=\mathbb{T}\setminus \{\sup\mathbb{T}\}$.
Otherwise, $\mathbb{T}^\kappa :=\mathbb{T}$.

\begin{definition}
Let $f:\mathbb{T} \rightarrow \mathbb{R}$ and $t \in
\mathbb{T}^\kappa$. The {\it delta derivative} of $f$ at $t$ is the
real number $f^{\Delta}(t)$ with the property that given any
$\varepsilon$ there is a neighborhood $U$ of $t$ (\textrm{i.e.},
$U=]t-\delta,t+\delta[\cap\mathbb{T}$ for some $\delta>0$) such that
\[|(f(\sigma(t))-f(s))-f^{\Delta}(t)(\sigma(t)-s)| \leq \varepsilon|\sigma(t)-s|\]
for all $s \in U$. We say that $f$ is {\it delta differentiable} on
$\mathbb{T}$ provided $f^{\Delta}(t)$ exists for all $t \in
\mathbb{T}^\kappa$.
\end{definition}

We shall often denote $f^\Delta(t)$ by $\frac{\Delta}{\Delta t}
f(t)$ if $f$ is a composition of other functions. The delta
derivative of a function $f:\mathbb{T} \rightarrow \mathbb{R}^n$ ($n
\in \mathbb{N}$) is a vector whose components are delta derivatives
of the components of $f$. For $f:\mathbb{T} \rightarrow X$, where
$X$ is an arbitrary set, we define $f^\sigma:=f\circ\sigma$.

For delta differentiable $f$ and $g$, the next formulas hold:

\begin{align*}
f^\sigma(t)&=f(t)+\mu(t)f^\Delta(t) \, ,\\
(fg)^\Delta(t)&=f^\Delta(t)g^\sigma(t)+f(t)g^\Delta(t)\\
&=f^\Delta(t)g(t)+f^\sigma(t)g^\Delta(t).
\end{align*}

\begin{remark}
If $\mathbb{T}=\mathbb{R}$, then $f:\mathbb{R} \rightarrow
\mathbb{R}$ is delta differentiable at $t \in \mathbb{R}$ if and
only if $f$ is differentiable in the ordinary sense at $t$. Then,
$f^{\Delta}(t)=f'(t)$. If $\mathbb{T}=\mathbb{Z}$, then
$f:\mathbb{Z} \rightarrow \mathbb{R}$ is always delta differentiable
at every $t \in \mathbb{Z}$ with $f^{\Delta}(t)=f(t+1)-f(t)$.
\end{remark}

Let $a,b \in \mathbb{T}$, $a<b$. We define the interval $[a,b]$ in
$\mathbb{T}$ by
$$[a,b]:=\{ t \in \mathbb{T}: a\leq t\leq b\}.$$
Open intervals, half-open intervals and unbounded intervals in
$\mathbb{T}$ are defined accordingly.

\begin{definition}
A function $F:\mathbb{T}\rightarrow\mathbb{R}$ is called a
\emph{delta antiderivative} of $f:\mathbb{T}\rightarrow\mathbb{R}$
provided
$$
F^{\Delta}(t)=f(t), \qquad  \forall t \in \mathbb{T}^\kappa.
$$
In this case we define the \emph{delta integral} of $f$ from $a$ to
$b$ ($a,b \in \mathbb{T}$) by
\begin{equation*}
\int_{a}^{b}f(t)\Delta t:=F(b)-F(a) \, .
\end{equation*}
\end{definition}

In order to present a class of functions that possess a delta
antiderivative, the following definition is introduced:

\begin{definition}
A function $f:\mathbb{T} \to \mathbb{R}$ is called
{\emph{rd-continuous}} if it is continuous at the right-dense points
in $\mathbb{T}$ and its left-sided limits exist at all left-dense
points in $\mathbb{T}$. A function $f:\mathbb{T} \to \mathbb{R}^n$
is {\emph{rd-continuous}} if all its components are rd-continuous.
\end{definition}

The set of all rd-continuous functions $f:\mathbb{T} \to
\mathbb{R}^n$ is denoted by $\mathrm{C}_{rd}(\mathbb{T},
\mathbb{R}^n)$, or simply by $\mathrm{C}_{rd}$. Similarly,
$\mathrm{C}^1_{rd}(\mathbb{T}, \mathbb{R}^n)$ and
$\mathrm{C}^1_{rd}$ will denote the set of functions from
$\mathrm{C}_{rd}$ whose delta derivative belongs to
$\mathrm{C}_{rd}$.

\begin{theorem}[\cite{Bohner-Peterson 2001}]
\label{antiderivada} Every rd-continuous function has a delta
antiderivative. In particular, if $a \in \mathbb{T}$, then the
function $F$ defined by
$F(t)= \int_{a}^{t}f(\tau)\Delta\tau$, $t \in \mathbb{T}$,
is a delta antiderivative of $f$.
\end{theorem}

\begin{theorem}[\cite{Bohner-Peterson 2001}]
\label{propriedades nabla integral} If $a,b,c \in \mathbb{T}$, $a
\le c \le b$, $\alpha \in \mathbb{R}$, and $f,g \in
C_{\textrm{rd}}(\mathbb{T}, \mathbb{R})$, then
\begin{enumerate}
\item $\displaystyle \int_{a}^{b}\left(f(t) + g(t) \right)
    \Delta t= \int_{a}^{b}f(t)\Delta t +
    \int_{a}^{b}g(t)\Delta t$;

\item $\displaystyle \int_{a}^{b} \alpha f(t)\Delta t =\alpha
    \int_{a}^{b}f(t)\Delta t$;

\item $\displaystyle \int_{a}^{b}  f(t)\Delta t = -
    \int_{b}^{a} f(t)\Delta t$;

\item $\displaystyle \int_{a}^{a}  f(t)\Delta t=0$;

\item $\displaystyle \int_{a}^{b}  f(t)\Delta t =
    \int_{a}^{c}  f(t)\Delta t + \int_{c}^{b} f(t)\Delta t$;

\item If $f(t)> 0$ for all $a \leq  t< b$, then $
    \displaystyle \int_{a}^{b}  f(t)\Delta t > 0$;

\item $\displaystyle \int_{a}^{b}f^\sigma(t)g^{\Delta}(t)\Delta t
=\left[(fg)(t)\right]_{t=a}^{t=b}-\int_{a}^{b}f^{\Delta}(t)g(t)\Delta
t$;

\item $\displaystyle \int_{a}^{b}f(t)g^{\Delta}(t)\Delta t
=\left[(fg)(t)\right]_{t=a}^{t=b}-\int_{a}^{b}f^{\Delta}(t)g^\sigma(t)\Delta
t$;

\item If $t \in \mathbb{T}^k$, then $\displaystyle \int_{t}^{\sigma(t)}  f(\tau)\Delta \tau =\mu(t)f(t)$.

\end{enumerate}
\end{theorem}

\begin{definition} If $a \in \mathbb{T}$,
$\sup \mathbb{T}=+\infty$ and $f \in C_{rd}([a,+\infty[, \mathbb{R})$,
then we define the improper delta integral by
$$
\int_{a}^{+\infty} f(t) \Delta t := \lim_{b\rightarrow+\infty}
\int_a^b f(t) \Delta t
$$
provided this limits exists (in $\overline{\mathbb{R}}
:=\mathbb{R}\cup\{-\infty, +\infty\}$). We say that the improper delta integral
converges if this limit is finite; otherwise, we say that the
improper delta integral diverges.
\end{definition}
In \cite{Bohner-Guseinov 2003} the reader may find many examples and
results involving delta improper integrals.

The following result will be very useful in the proof of our main
result (Theorem~\ref{main result}).

\begin{theorem}[\cite{Serge Lang}]
\label{Serge Lang}
Let $S$ and $T$ be subsets of a  normed vector space. Let $f$ be a
map defined on $T \times S$, having values in some complete normed
vector space. Let $v$ be adherent to $S$ and $w$ adherent to $T$.
Assume that:
\begin{enumerate}
\item $\lim_{x\rightarrow v} f(t,x)$ exists for each $t \in T$;

\item $\lim_{t\rightarrow w} f(t,x)$ exists uniformly for  $x \in S$.
\end{enumerate}
Then the limits
$\lim_{t\rightarrow w} \lim_{x\rightarrow v} f(t,x)$,
$\lim_{x\rightarrow v} \lim_{t\rightarrow w}f(t,x)$,
and $\lim_{(t,x)\rightarrow (v,w)} f(t,x)$
all exist and are equal.
\end{theorem}


\section{Main results}

In this section we assume that $\mathbb{T}$ is a time scale
such that $\sup \mathbb{T} = +\infty$. In what follows we will
suppose that $a, T, T^\prime\in \mathbb{T}$ are such that $T > a$
and $T^\prime >a$. By $\partial_2 L$ and
$\partial_3 L$ we denote, respectively, the partial
derivative of $L(\cdot,\cdot,\cdot)$ with respect to its second and third
argument. Let us consider the following variational problem
on $\mathbb{T}$:

\begin{equation}
\label{problem}
\begin{gathered}
\int_{a}^{+\infty} L(t,x^{\sigma}(t), x^{\Delta}(t)) \Delta t  \longrightarrow \max \\
x \in C^{1}_{rd}(\mathbb{T},\mathbb{R}^n)\\
x(a)=x_a
\end{gathered}
\end{equation}
where $(u,v)\rightarrow L(t,u,v)$ is a $C^1(\mathbb{R}^{2n}, \mathbb{R})$ function for any $t \in \mathbb{T}$,
$\partial_3 L(t, x^\sigma(t), x^\Delta (t))$ is delta differentiable for all $x \in C^{1}_{rd}(\mathbb{T},\mathbb{R}^n)$,
$n \in \mathbb{N}$, and $x_a \in \mathbb{R}^n$.

\begin{definition}
We say that $x$ is an admissible path for problem (\ref{problem}) if
and only if $x \in C^{1}_{rd}(\mathbb{T},\mathbb{R}^n)$ and
$x(a)=x_a$.
\end{definition}

We use the following notion as our
optimality criteria.

\begin{definition}[weak maximality]
\label{def:weakMax}
We say that $x_{\ast}$ is weakly maximal to problem
(\ref{problem}) if and only if $x_{\ast}$ is an admissible path and
$$
\lim_{T\rightarrow+\infty}\inf_{T^\prime \geq
T}\int_{a}^{T^\prime}[L(t, x^{\sigma}(t), x^{\Delta}(t))- L(t,
x_{\ast}^{\sigma}(t), x_{\ast}^{\Delta}(t)]\Delta t \le 0
$$
for all admissible path $x$.
\end{definition}

\begin{lemma}
\label{Fundamental lemma}
Let $g \in C_{rd}(\mathbb{T}, \mathbb{R})$.
Then,
$$
\lim_{T\rightarrow+\infty}\inf_{T^\prime \geq T}\int_{a}^{T^\prime} g(t) \eta^{\sigma}(t) \Delta t=0
\ \ \ \mbox{for all} \ \ \eta \in C_{rd}(\mathbb{T}, \mathbb{R}) \ \ \ \mbox{such that} \ \  \eta(a)=0
$$
if and only if $g(t)=0$ on $[a, +\infty[$.
\end{lemma}

\begin{proof}
The implication "$\Leftarrow$" is obvious.
Let us prove the implication "$\Rightarrow$". Suppose, by contradiction,
that $g(t)\not \equiv 0$. Let $t_0$ be a
point on  $[a, +\infty[$ such that $g(t_0)\neq 0$; suppose, without
loss of generality, that $g(t_0)>0$.

\emph{Case I}. If $t_0$ is right-dense, then $g$ is also positive in
$[t_0, t_1]$ for some $t_1>t_0$.
Define
$$\eta(t)= \left\{
\begin{array}{lcl}
(t-t_0)(t_1-t)  & & t \in [t_0,t_1]\\
0 & & \mbox{otherwise}\, .
\end{array} \right.
$$
Then $$\lim_{T\rightarrow+\infty}\inf_{T^\prime \geq
T}\int_{a}^{T^\prime} g(t) \eta^{\sigma}(t) \Delta t=
\int_{t_0}^{t_1}g(t) \eta^{\sigma}(t) \Delta t >0$$
which is a contradiction.

\emph{Case II}. Suppose that $t_0$ is right-scattered.
\begin{enumerate}
\item If $\sigma(t_0)$ is right-scattered, define
$$\eta(t)= \left\{
\begin{array}{lcl}
g(t_0)  & & t=\sigma(t_0)\\
0 & & \mbox{otherwise}
\end{array} \right..
$$
Then
$$\lim_{T\rightarrow+\infty}\inf_{T^\prime \geq T}\int_{a}^{T^\prime} g(t) \eta^{\sigma}(t) \Delta t=
\int_{t_0}^{\sigma(t_0)}g(t) \eta^{\sigma}(t) \Delta t =
\mu(t_0)g(t_0)g(t_0)>0$$
which is a contradiction.

\item Suppose that $\sigma(t_0)$ is right-dense. Two situations may occur:
\begin{enumerate}
\item $g(\sigma(t_0))\neq 0$;

\item $g(\sigma(t_0))= 0$.
\end{enumerate}
In case $(a)$, we can assume, without loss of generality, that
$g(\sigma(t_0))>0$. Then $g$ is also positive in $[\sigma(t_0),
t_2]$ for some $t_2 > \sigma(t_0)$.
Define
$$\eta(t)= \left\{
\begin{array}{lcl}
(t-\sigma(t_0))(t_2-t)  & & t \in [\sigma(t_0),t_2]\\
0 & & \mbox{otherwise}
\end{array} \right..
$$
In this case
$$\lim_{T\rightarrow+\infty}\inf_{T^\prime \geq T}\int_{a}^{T^\prime} g(t) \eta^{\sigma}(t) \Delta t=
\int_{\sigma(t_0)}^{t_2}g(t) \eta^{\sigma}(t) \Delta t >0$$
which is a contradiction.

Suppose we are in case $(b)$. Two situations may happen:

$(i)$ $g(t)=0$ on $[\sigma(t_0),t_3]$ for some $t_3 > \sigma(t_0);$

$(ii)$ $\forall t_3 > \sigma(t_0) \  \exists t \in [\sigma(t_0), t_3] \ g(t)\neq 0$.

In case $(i)$ define

$$\eta(t)= \left\{
\begin{array}{lcl}
g(t_0)  & & t = \sigma(t_0)\\
\varphi(t) & & t \in ]\sigma(t_0), t_3]\\
0 & & \mbox{otherwise}
\end{array} \right.
$$
for some function  $\varphi \in C_{rd}$ satisfying the conditions
$\varphi(t_3)=0$ and $\varphi(\sigma(t_0))=g(t_0)$.
It follows that
$$\lim_{T\rightarrow+\infty}\inf_{T^\prime \geq T}\int_{a}^{T^\prime} g(t) \eta^{\sigma}(t) \Delta t=
\int_{t_0}^{\sigma(t_0)}g(t) \eta^{\sigma}(t) \Delta t=
\mu(t_0)g(t_0)g(t_0) >0$$
which is a contradiction.

Suppose we are in case $(ii)$. Since $\sigma(t_0)$ is right-dense,  there exists a strictly
decreasing sequence $S=\{s_{k}: k\in \mathbb{N}\}\subseteq \mathbb{T}$
such that  $\lim_{k \rightarrow \infty}s_{k}=\sigma(t_0)$ and
$g(s_k)\neq 0,  \ \forall k \in \mathbb{N}$. If there exists
a right-dense $s_{k}$, then go to
\emph{Case I} with $t_0:=s_{k}$ (and we get a contradiction).
If all points of the sequence are right-scattered, then go to Case II
with $t_0:=s_j$ for some $j \in \mathbb{N}$. Since $\sigma(t_0)$
is right-scattered, we are in situation 1 and we obtain a contradiction.
\end{enumerate}
Therefore, we may conclude that $g\equiv 0$ on $[a,+\infty[$.
\end{proof}

\begin{theorem}
\label{main result}
Suppose that the optimal path to problem (\ref{problem}) exists and
is given by $x_{\ast}$. Let $p \in
C^1_{rd}(\mathbb{T},\mathbb{R}^n)$ be such that $p(a)=0$. Define
$$
\begin{array}{lcl}
 A(\varepsilon, T^\prime)& := & \displaystyle \int_{a}^{T^\prime}
\frac{L(t, x_{\ast}^{\sigma}(t) + \varepsilon p^{\sigma}(t),
x_{\ast}^{\Delta}(t) + \varepsilon p^{\Delta}(t))
- L(t, x_{\ast}^{\sigma}(t), x_{\ast}^{\Delta}(t))}{\varepsilon} \Delta t\\
& & \\
V(\varepsilon, T)& := & \displaystyle \inf_{T^\prime \geq
T}\int_{a}^{T^\prime} \left [L(t, x_{\ast}^{\sigma}(t) + \varepsilon
p^{\sigma}(t), x_{\ast}^{\Delta}(t) + \varepsilon p^{\Delta}(t))
- L(t, x_{\ast}^{\sigma}(t), x_{\ast}^{\Delta}(t))\right] \Delta t\\
& & \\
V(\varepsilon)&:= & \displaystyle \lim_{T\rightarrow+\infty}
V(\varepsilon, T).
\end{array}
$$

Suppose that

\begin{enumerate}

\item $\displaystyle \lim_{\varepsilon \rightarrow 0} \frac{V(\varepsilon, T) }{\varepsilon}$ exists for all $T$;

\item $\displaystyle \lim_{T\rightarrow+\infty}\frac{V(\varepsilon, T) }{\varepsilon}$ exists uniformly for $\varepsilon$;

\item For every $T^\prime > a$, $T > a$,
and $\varepsilon\in \mathbb{R}\setminus\{0\}$,
there exists a sequence $\left(A(\varepsilon, T^\prime_n)\right)_{n \in \mathbb{N}}$
such that
$$
\displaystyle \lim_{n \rightarrow +\infty} A(\varepsilon, T^\prime_n)
= \displaystyle \inf_{T^\prime \geq T} A(\varepsilon, T^\prime)
$$
uniformly for $\varepsilon$.
\end{enumerate}
Then $x_{\ast}$ satisfies the Euler-Lagrange equation
\begin{equation}
\label{E-L equation}
\frac{\Delta}{\Delta t} \partial_3 L (t, x^{\sigma}(t),
x^{\Delta}(t))= \partial_2 L (t, x^{\sigma}(t), x^{\Delta}(t) ), \ \
\ \ \ \forall t \in [a, +\infty[
\end{equation}
and the tranversality condition
\begin{equation}
\label{tranversality}
\displaystyle \lim_{T\rightarrow+\infty} \inf_{T^\prime \geq T}
\partial_3 L (T^\prime, x^{\sigma}(T^\prime), x^{\Delta}(T^\prime)) x(T^\prime)=0.
\end{equation}
\end{theorem}

\begin{remark}
Similarly to the classical context $\mathbb{T} = \mathbb{R}$ \cite{Nitta et all},
hypotheses 1, 2, and 3 of Theorem~\ref{main result} are impossible to be
verified \emph{a priori} because $x_{\ast}$ is unknown. In practical terms
such hypotheses are assumed to be true and conditions
\eqref{E-L equation} and \eqref{tranversality} are applied heuristically
to obtain a \emph{candidate}. If such a candidate is, or not, a solution
to the problem is a different question that always require further analysis
(see Examples~\ref{ex:neg} and \ref{ex:pos}).
\end{remark}

\begin{proof}
Using our notion of weak maximality, if $x_{\ast}$ is optimal, then
$V(\varepsilon) \leq 0$ for every $\varepsilon \in \mathbb{R}$.
Since $V(0)=0$, then 0 is an extremal point of $V$.
If $V$ is differentiable at $t=0$, we may conclude that
$V^\prime(0)=0$.
We now note that
$$
\begin{array}{lcl}
V^\prime(0) & = & \displaystyle
\lim_{\varepsilon \rightarrow 0} \frac{V(\varepsilon)}{\varepsilon}
= \displaystyle \lim_{\varepsilon \rightarrow 0}
\lim_{T\rightarrow+\infty}\frac{V(\varepsilon, T) }{\varepsilon}\\
& = & \displaystyle  \lim_{T\rightarrow+\infty}
\lim_{\varepsilon \rightarrow 0} \frac{V(\varepsilon, T) }{\varepsilon}
\ \ \ \ \ \ (\mbox{by hypothesis \emph{1} and \emph{2} and Theorem \ref{Serge Lang})}\\
& = & \displaystyle  \lim_{T\rightarrow+\infty}
\lim_{\varepsilon \rightarrow 0} \displaystyle \inf_{T^\prime \geq T} A(\varepsilon, T^\prime)\\
& = & \displaystyle  \lim_{T\rightarrow+\infty}
\lim_{\varepsilon \rightarrow 0} \lim_{n \rightarrow +\infty} A(\varepsilon, T^\prime_n)
\ \ \ \ \ \ (\mbox{by hypothesis \emph{3}) }\\
& = & \displaystyle  \lim_{T\rightarrow+\infty}
\lim_{n \rightarrow +\infty} \lim_{\varepsilon \rightarrow 0} A(\varepsilon, T^\prime_n)
\ \ \ \ \ \ (\mbox{by hypothesis \emph{3} and Theorem \ref{Serge Lang}) }\\
& = & \displaystyle  \lim_{T\rightarrow+\infty}
\inf_{T^\prime \geq T}  \lim_{\varepsilon \rightarrow 0} A(\varepsilon, T^\prime)
\ \ \ \ \ \ (\mbox{by hypothesis \emph{3})}\\
& = & \displaystyle  \lim_{T\rightarrow+\infty}
\inf_{T^\prime \geq T}  \lim_{\varepsilon \rightarrow 0} \displaystyle
\int_{a}^{T^\prime} \frac{L(t, x_{\ast}^{\sigma}(t) + \varepsilon
p^{\sigma}(t), x_{\ast}^{\Delta}(t) + \varepsilon p^{\Delta}(t))
- L(t, x_{\ast}^{\sigma}(t), x_{\ast}^{\Delta}(t))}{\varepsilon} \Delta t\\
& = & \displaystyle  \lim_{T\rightarrow+\infty}  \inf_{T^\prime \geq
T} \displaystyle \int_{a}^{T^\prime} \lim_{\varepsilon \rightarrow
0} \frac{L(t, x_{\ast}^{\sigma}(t) + \varepsilon p^{\sigma}(t),
x_{\ast}^{\Delta}(t) + \varepsilon p^{\Delta}(t))
- L(t, x_{\ast}^{\sigma}(t), x_{\ast}^{\Delta}(t))}{\varepsilon} \Delta t\\
& = & \displaystyle  \lim_{T\rightarrow+\infty}  \inf_{T^\prime \geq
T} \displaystyle \int_{a}^{T^\prime} \left[\partial_2 L(t,
x_{\ast}^{\sigma}(t), x_{\ast}^{\Delta}(t)) p^{\sigma}(t)+
\partial_3 L(t, x_{\ast}^{\sigma}(t), x_{\ast}^{\Delta}(t)) p^{\Delta}(t)\right]  \Delta t\\
& = & \displaystyle  \lim_{T\rightarrow+\infty}  \inf_{T^\prime \geq
T} \displaystyle \{\int_{a}^{T^\prime} \left[\partial_2 L(t,
x_{\ast}^{\sigma}(t), x_{\ast}^{\Delta}(t)) p^{\sigma}(t)-
\frac{\Delta}{\Delta t} \partial_3 L(t, x_{\ast}^{\sigma}(t), x_{\ast}^{\Delta}(t)) p^{\sigma}(t) \right] \Delta t\\
& &  \ \ \ \ +
\partial_3 L(T^\prime, x_{\ast}^{\sigma}(T^\prime), x_{\ast}^{\Delta}(T^\prime)) p(T^\prime)\}
\ \ \ \ \ \ (\mbox{by item \emph{8} of  Theorem \ref{propriedades nabla integral} and $p(a)=0$)}.\\
\end{array}
$$
Hence we may conclude that
\begin{equation}\label{eq 1}
\displaystyle  \lim_{T\rightarrow+\infty}  \inf_{T^\prime \geq
T}\left\{ \displaystyle \int_{a}^{T^\prime} \left(\partial_2
L(\bullet)- \frac{\Delta}{\Delta t} \partial_3 L(\bullet)\right)
p^{\sigma}(t)  \Delta t +
\partial_3 L(T^\prime, x_{\ast}^{\sigma}(T^\prime),
x_{\ast}^{\Delta}(T^\prime)) p(T^\prime)\right\}=0
\end{equation}
\\
where we denote $(\bullet):=(t, x_{\ast}^{\sigma}(t),
x_{\ast}^{\Delta}(t))$.
Since $(\ref{eq 1})$ holds for all $p \in C^1_{rd}$ such that
$p(a)=0$, then, in particular, $(\ref{eq 1})$ holds for $p$
satisfying also $p(T^\prime)=0$. Therefore,
\begin{equation}\label{eq 2}
\displaystyle  \lim_{T\rightarrow+\infty}  \inf_{T^\prime \geq T}
\displaystyle \int_{a}^{T^\prime} \left(\partial_2 L(\bullet)-
\frac{\Delta}{\Delta t} \partial_3 L(\bullet)\right) p^{\sigma}(t)
\Delta t =0.
\end{equation}
Denote
$$\partial_2 L=\left(\frac{\partial L}{\partial x_1},
\cdots,\frac{\partial L}{\partial x_n}\right)  \ \ \ \ \mbox{and}\ \ \ \
\partial_3 L=\left(\frac{\partial L}{\partial y_1},
\cdots,\frac{\partial L}{\partial y_n}\right).$$
Choosing $p=(p_1, \cdots, p_n)$ such that $p_2,\ldots, p_n\equiv 0$,
we obtain from (\ref{eq 2}) that
$$
\displaystyle  \lim_{T\rightarrow+\infty}  \inf_{T^\prime \geq T}
\displaystyle \int_{a}^{T^\prime} \left(\frac{\partial L}{\partial
x_1}(\bullet)- \frac{\Delta}{\Delta t} \frac{\partial L}{\partial
y_1}(\bullet)\right) p_1^{\sigma}(t)  \Delta t =0.
$$
Using Lemma~\ref{Fundamental lemma} we conclude that
$$
\frac{\partial L}{\partial x_1}(\bullet)- \frac{\Delta}{\Delta t}
\frac{\partial L}{\partial y_1}(\bullet)=0,  \ \ \ \ \ \forall t \in
[a,+\infty[.
$$
This procedure can be similarly used for the other coordinates and
we obtain the Euler-Lagrange equations:
$$
\frac{\partial L}{\partial x_i}(\bullet)- \frac{\Delta}{\Delta t}
\frac{\partial L}{\partial y_i}(\bullet)=0,  \ \ \ \ \ \forall t \in
[a,+\infty[
$$
for $i=2, 3, \ldots, n$.
These $n$ Euler-Lagrange equations can be written in the condensed
form
\begin{equation}
\label{E-L equation 1}
\partial_2 L(t, x_{\ast}^{\sigma}(t), x_{\ast}^{\Delta}(t))-
\frac{\Delta}{\Delta t} \partial_3 L(t, x_{\ast}^{\sigma}(t),
x_{\ast}^{\Delta}(t))=0, \ \ \ \ \ \forall t \in [a,+\infty[.
\end{equation}
The Euler-Lagrange equation (\ref{E-L equation 1}) and equation
(\ref{eq 1}) shows that
\begin{equation}
\label{tranversality 1}
\displaystyle  \lim_{T\rightarrow+\infty}  \inf_{T^\prime \geq
T}\partial_3 L(T^\prime, x_{\ast}^{\sigma}(T^\prime),
x_{\ast}^{\Delta}(T^\prime)) p(T^\prime)=0.
\end{equation}
Next we consider a special curve $p$ defined by
$$p(t)=\alpha(t) x_{\ast}(t), \ \ \ \ \forall t \in [a, +\infty[$$
where $\alpha: [a, +\infty[ \rightarrow \mathbb{R}$ is a $C^1_{rd}$
function satisfying $\alpha (a)=0$ and there exists $T_0\in
\mathbb{T}$ such that $\alpha(t)=\beta \in
\mathbb{R}\setminus\{0\}$, for all $t> T_0$.
By equation (\ref{tranversality 1}) we conclude that
$$
\begin{array}{lcl}
0&= & \displaystyle  \lim_{T\rightarrow+\infty}  \inf_{T^\prime \geq T}\partial_3
L(T^\prime, x_{\ast}^{\sigma}(T^\prime), x_{\ast}^{\Delta}(T^\prime)) \alpha(T^\prime)x_{\ast}(T^\prime)\\
&= &  \displaystyle  \lim_{T\rightarrow+\infty}  \inf_{T^\prime \geq T}\partial_3
L(T^\prime, x_{\ast}^{\sigma}(T^\prime), x_{\ast}^{\Delta}(T^\prime)) \beta x_{\ast}(T^\prime) \\
\end{array}
$$
and therefore
$$
\displaystyle  \lim_{T\rightarrow+\infty}  \inf_{T^\prime \geq T}\partial_3
L(T^\prime, x_{\ast}^{\sigma}(T^\prime), x_{\ast}^{\Delta}(T^\prime)) )x_{\ast}(T^\prime)=0,
$$
proving that $x_{\ast}$ satisfies the transversality condition
(\ref{tranversality}).
\end{proof}


\section{Illustrative examples}

\begin{ex}
\label{ex:neg}
Consider the problem
\begin{equation*}
\displaystyle \int_{a}^{+\infty} \left[(x^\sigma(t)-\alpha)^2
+\beta x^\Delta(t)\right]\ \Delta t \longrightarrow \max \, ,
\quad x(a)=\alpha \, ,
\end{equation*}
where $\alpha>0$ and $\beta>0$. Since
$L(t,x^\sigma,x^\Delta)=(x^\sigma-\alpha)^2+\beta x^\Delta$
we have
$\partial_2L=2(x^\sigma-\alpha)$ and
$\partial_3L=\beta$.
From Theorem~\ref{main result} the Euler-Lagrange equation is
\begin{equation*}
2(x^\sigma(t)-\alpha)=0, \ \forall t \in [a,+\infty[ \, ,
\end{equation*}
that is, $x^\sigma(t)=\alpha, \ \forall t \in [a,+\infty[$. As $x(a)=\alpha$, we have
$x(t)=\alpha, \ \forall t \in [a,+\infty[$. Observe that the transversality condition
\eqref{tranversality} is not satisfied, because
\begin{equation*}
\lim_{T\rightarrow+\infty} \inf_{T^\prime \geq T}\partial_3 L
(T^\prime, x^{\sigma}(T^\prime), x^{\Delta}(T^\prime))
x(T^\prime)=\beta \alpha>0.
\end{equation*}
The reason why we obtain this contradiction is that assumptions of
Theorem~\ref{main result} are violated. Consider $L(t, x^{\sigma}(t)
+ \varepsilon p^{\sigma}(t), x^{\Delta}(t) + \varepsilon
p^{\Delta}(t)) - L(t, x^{\sigma}(t), x^{\Delta}(t))$. Substituting
$x(t)=\alpha$ into it, we have
\begin{equation*}
L(t, x^{\sigma}(t) + \varepsilon p^{\sigma}(t), x^{\Delta}(t) +
\varepsilon p^{\Delta}(t)) - L(t, x^{\sigma}(t),
x^{\Delta}(t))=\varepsilon^2(p^\sigma(t))^2+\beta\varepsilon
 p^\Delta(t).
\end{equation*}
Hence,
\begin{equation*}
\frac{V(\varepsilon,T)}{\varepsilon}=\inf_{T^\prime \geq
T}\int_{a}^{T^\prime}\frac{\varepsilon^2(p^\sigma(t))^2+\beta\varepsilon
 p^\Delta(t)}{\varepsilon}\Delta
t=\inf_{T^\prime \geq
T}\int_{a}^{T^\prime}\left(\varepsilon(p^\sigma(t))^2+\beta
 p^\Delta(t)\right)\Delta t.
\end{equation*}
Choosing $p$ such that $p(a)=0$ and there exists $T_0>a$ so that
$p(t)=c>0$ for $t\geq T_0$, we obtain
\begin{equation*}
\frac{V(\varepsilon,T)}{\varepsilon}=\int_{a}^{T}\varepsilon
c^2\Delta t+\beta c.
\end{equation*}
Therefore, assumption~2 of Theorem~\ref{main result} is violated.
\end{ex}

\begin{ex}
\label{ex:pos}
Consider the problem
\begin{equation}
\label{problem:ex2}
\displaystyle \int_{0}^{+\infty} -\sqrt{1+ (x^\Delta(t))^2}\ \Delta t \longrightarrow \max \, ,
\quad x(0)=A \, .
\end{equation}
Since
\begin{equation*}
L(t,x^\sigma,x^\Delta)=-\sqrt{1+ (x^\Delta)^2} \, ,
\end{equation*}
we have
\begin{equation*}
\partial_3L=-\frac{x^{\Delta}}{\sqrt{1+(x^{\Delta})^2}}, \quad \partial_2L=0.
\end{equation*}
Using the Euler-Lagrange equation \eqref{E-L equation} we obtain
\begin{equation*}
\tilde{x}^{\Delta}(t)=d\sqrt{1+(\tilde{x}^{\Delta}(t))^2}, \ \forall t \in [0,+\infty[
\end{equation*}
for some $d\in \mathbb{R}$. Solving the latter equation with initial
condition $x(0)=A$ we obtain $\tilde{x}(t)=\alpha t+A$, where
$\alpha \in \mathbb{R}$. In order to determine $\alpha$ we use the
tranversality condition \eqref{tranversality}, which can be
rewritten as
\begin{equation*}
\lim_{T\rightarrow+\infty} \inf_{T^\prime \geq
T}-\frac{\alpha}{\sqrt{1+\alpha^{2}}}(\alpha T^\prime+A)=0.
\end{equation*}
Hence, $\alpha=0$ and $x_{\ast}(t)=A$ is a \emph{candidate} to be a maximizer.
Observe that
\begin{multline*}
\lim_{T\rightarrow+\infty}\inf_{T^\prime \geq
T}\int_{0}^{T^\prime}[L(t, x^{\sigma}(t), x^{\Delta}(t))
- L(t,x_{\ast}^{\sigma}(t), x_{\ast}^{\Delta}(t)]\Delta t\\
= \lim_{T\rightarrow+\infty}\inf_{T^\prime \geq
T}\int_{0}^{T^\prime} \left(
1 - \sqrt{1 + (x^\Delta(t))^2}\right) \Delta t \le 0
\end{multline*}
for every admissible $x$. Therefore, by Definition~\ref{def:weakMax}
we have that $x_{\ast}(t)=A$ is indeed the solution to problem \eqref{problem:ex2}.
\end{ex}


\section{Conclusion and future work}

We considered problems of the calculus of variations
on unbounded time scales. Main result provides
a new transversality condition. Examples illustrating
the application of the new necessary optimality
conditions are given in detail.
In the particular case $\mathbb{T} = \mathbb{Z}$
our transversality condition gives
the discrete time condition obtained by Michel in \cite{Michel};
for the continuous time case, \textrm{i.e.}, for $\mathbb{T} = \mathbb{R}$,
we obtain the result by Kamihigashi \cite{Kami}.
Recently, Okumura et al. \cite{Nitta et all} generalized
the results of Kamihigashi to higher order differential problems.
The question of obtaining necessary optimality conditions
that extend the results of \cite{Nitta et all}
to higher-order infinite horizon problems on time scales remains
an interesting open question. While clear that the Euler-Lagrange
equations proved in \cite{comRui:TS:Lisboa07,NataliaHigherOrderNabla}
remain valid in the infinite horizon case, the generalization
of our transversality condition \eqref{tranversality}
to higher-order variational problems on time scales
is a non-trivial question requiring further investigations.


\section*{Acknowledgments}

This work was partially supported by the R\&D unit
``Centre for Research on Optimization and Control'' (CEOC)
of the University of Aveiro,
cofinanced by the European Community Fund FEDER/POCI 2010.
Agnieszka Malinowska is on leave of absence
from Bia{\l}ystok University of Technology (BUT).
She was also supported by BUT,
via a project of the Polish Ministry of Science
and Higher Education ``Wsparcie miedzynarodowej mobilnosci
naukowcow''.



\end{document}